\documentclass[11pt,a4paper]{article}

\usepackage{amsfonts,amsthm,amsmath,graphicx,enumerate}
\usepackage{graphics,graphicx,amssymb,eucal,amsmath}
\theoremstyle{plain}
\newtheorem{thm}{Theorem}[section]

\newtheorem{pro}[thm]{Proposition}

\newtheorem{defn}[thm]{Definition}

\newtheorem{rem}[thm]{Remark}

\makeatletter \@addtoreset{equation}{section} \makeatother

\newcommand{\pam}{\par\medskip}
\newcommand{\pas}{\par\smallskip}

\def\mm{\medskip\\}

\def\beq#1{\begin{equation}\label{#1}}
\def\eeq{\end{equation}}

\title{Geometric properties of \\reliability polynomials}
\author{Z. A. H. Hassan, C. Udriste, V. Balan}

\begin{document}
\maketitle

\begin{abstract}
Geometric modeling of multivariate reliability polynomials is
based on algebraic hypersurfaces, constant level sets, rulings etc.
The solved basic problems are: (i) find the reliability polynomial using the Maple and Matlab software environment;
(ii) find restrictions of reliability polynomial via equi-reliable components;
(iii) how should the reliability components linearly depend on time, so that the reliability
of the system be linear in time?
The main goal of the paper is to find geometric
methods for analysing the reliability of electric systems used
inside aircrafts.
\end{abstract}

{\bf Keywords}: reliability polynomial; ruled hypersurfaces; electric systems inside aircrafts.

{\bf 2010 Mathematics Subject Clasification}: 60K10, 62N05, 90B25.

\section{Introduction}
During the decades following the war, many research laboratories
    and universities developed and initiated programs to study life
    testing and reliability problems \cite{Fw, AB}.
Numerous such research topics focus on the study of different
    types of reliability systems \cite{HZ, HZ1, HZ2}, like serial, parallel,
    serial-parallel, parallel-serial, and complex,
    which have considerable impact on different life fields \cite{Fw}-\cite{RS}.
Since there exist different important available systems, the
researchers attempted to find more than one method to solve these complex systems,
    and determine the optimal ones \cite{Fw, HZ1, IY}.\par
In the present work, we change the classical view, by trying to
get information from the differential geometry naturally related to the stochastic models.
Of particular interest is the study of reliability hypersurface
    and establishing the number of straight lines situated on this set. For further ideas, see \cite{CU}.
%============================================================================================
\section{Some definitions and basic terminology}
We shall present first the concepts in network topology and
    in graph theory which are needed
    to calculate the network reliability \cite{Fw}-\cite{SA}.\pas
\begin{defn}
    A graph $G = (V,E)$,
    where $V$ is the set of vertices (or nodes) and $E$ the set of
    edges (or arcs), is called a network.\pam
\end{defn}

\noindent
{\bf The network model}.
We describe our system as a directed network consisting of nodes
    and arcs, as illustrated in Fig. 1.
One node is considered as the \emph{source} (node $A$ in the
figure), and a second node is considered as a \emph{sink} (node $D$).
Each component of the network is identified as an arc passing from
    one node to another.
The arcs are numbered for identification. A {\em failure of a
component} is equivalent to an arc being removed or
    cut out from the network.
The system is {\em successful} if there exists a valid path from
the source to the sink. The system is said to be {\em failed} if no such path exists. The
{\em reliability of the system} is the probability that there
    exist one or more successful paths from the source to the sink \cite{Lj, ML}.

\medskip

\centerline{\includegraphics[scale=.34]{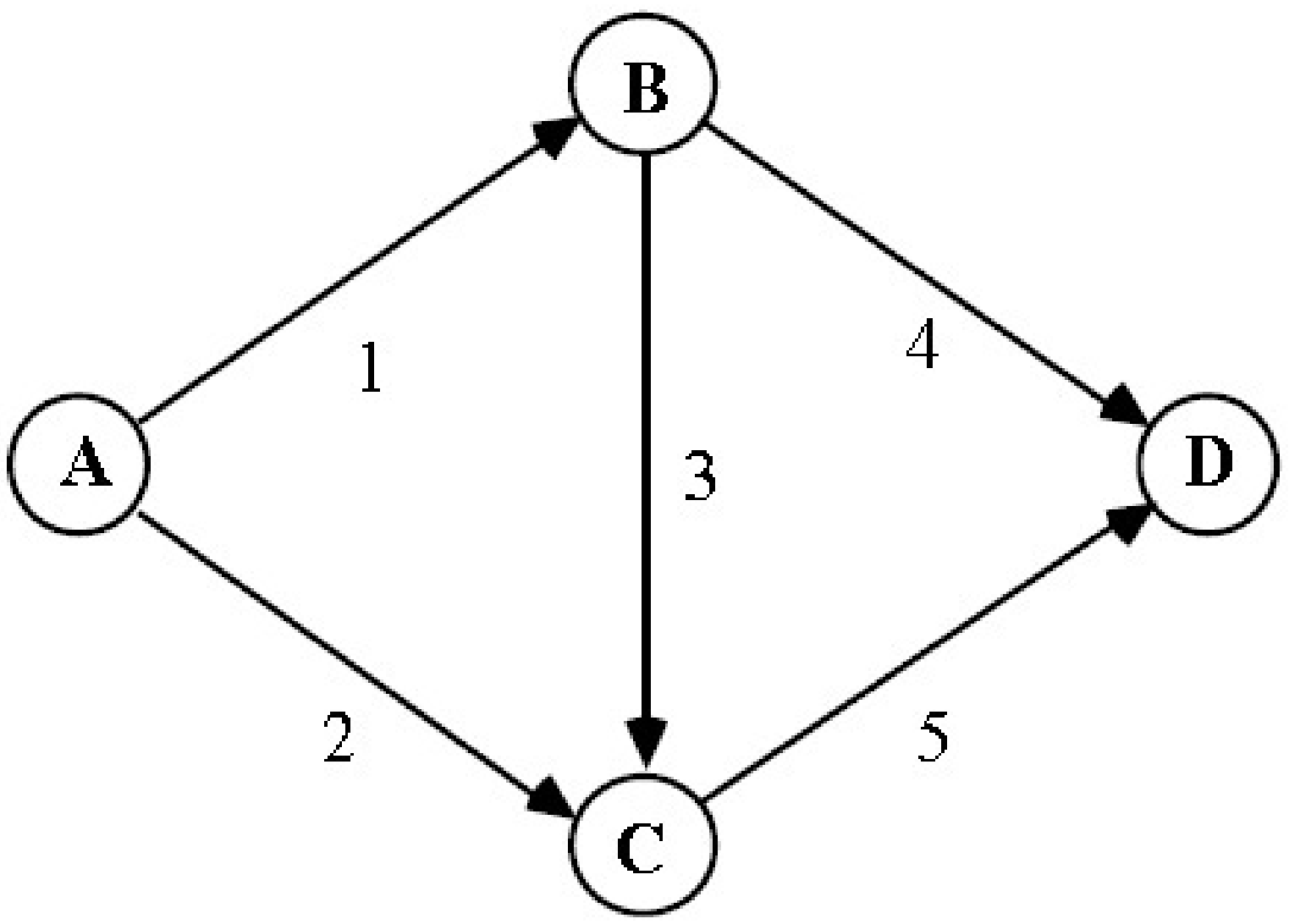}}

\centerline{{\sc Fig. 1}. A bridge network. }

\medskip

\noindent
\begin{defn}
 A set of components is called a {cut}
if, when all the components in this set fail, the
    system will fail, even if all other components are successful. \pas
\end{defn}

\begin{defn}
A cut, such that any removal of
one component from it causes the resulting
    set do not be a cut, is called a {minimal cut}.
 \end{defn}

    The set of all components is a cut.
    In the network a minimal cut breaks all
    simple paths from the source to the sink. In Fig. 1, we observe
    that the minimal cuts are: $\{1, 2\}$, $\{1, 5\}$, $\{2, 3, 4\}$,
    and $\{4, 5\}$.
%============================================================================================
\section{Complex reliability systems (network model)}

We introduce a graphical network model in which it is possible to
    determine whether a system is working correctly by determining
    whether a successful path exists in the system.
The system fails when no such path exists.\pas

The system in Fig. 2 cannot be split into a group of series
    and parallel systems.

    \medskip

\centerline{\includegraphics[scale=.44]{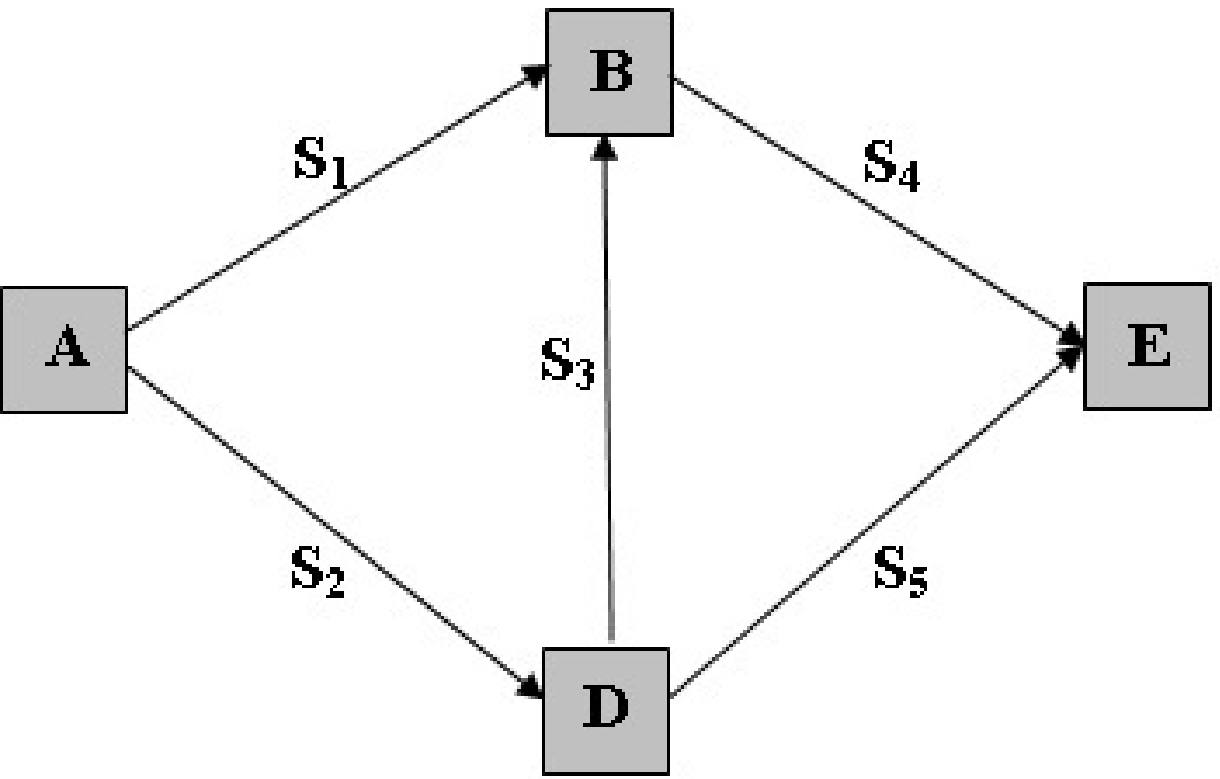}}

\centerline{{\sc Fig. 2}. A complex system (network model). }

\medskip

This is primarily due to the fact that the components $A$ and $D$
    each allow two paths emerging from them, whereas $B$ has only one;
    $S_1$, $S_2$, $S_3$, $S_4$ and $S_5$ are called subsystems
    or arcs.

\subsection{Minimal cut method}

There exist several methods for obtaining the reliability
of a complex system, as, for example, {\em minimal cut method}.
The minimal cut method is proper for systems which are connected in the form of a bridge.
When we apply this method to the system in Fig. 2, we should pursue the
following steps:

\begin{enumerate}[a)]

\item we enumerate all the minimal cut-sets in the system;

\item the failure of all components in a minimal
    cut-set causes system failure;

\item this implies \emph{parallel} connections among these
components;

\item each minimal cut set determines the system failure;

\item this implies \emph{series} connections among the minimal cut
sets;

\item we draw an equivalent system and use the
\emph{parallel/seies method}
    to compute the system reliability.
\end{enumerate}

\begin{thm}If $S_{1}, S_{2}, S_{3}, S_{4}, S_{5}$ are arcs (paths) in a bridge system (Fig. 2),
then the  reliability $R_{Mc}(t)$ of all system is
$$R_{Mc}(t)=R_1(t)R_4(t) + R_2(t)R_5(t) + R_2(t)R_3(t)R_4(t) \eqno(3.1)$$
$$
 - R_1(t)R_2(t)R_3(t)R_4(t) - R_1(t)R_2(t)R_4(t)R_5(t)
$$
$$
 - R_2(t)R_3(t)R_4(t)R_5(t) + R_1(t)R_2(t)R_3(t)R_4(t)R_5(t).
$$
\end{thm}

\begin{proof} By using minimal cut method, we have

$$
\hbox{Minimal cut-set} = \{ (S_{1}, S_{2}), (S_{4}, S_{5}),(S_{2}, S_{4}), (S_{1}, S_{3}, S_{5}) \},
$$
and then Fig. 2 can be replaced by Fig. 3, which will represent the reliability of a
parallel-series system ~\cite{HZ}, as follows:

\medskip

\centerline{\includegraphics[scale=.44]{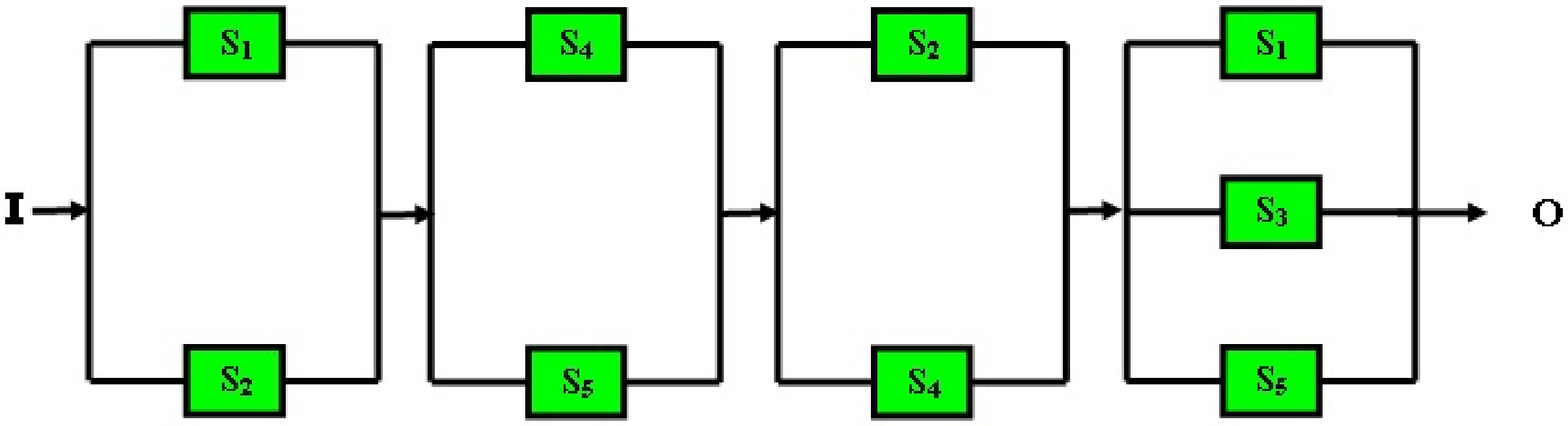}}

\centerline{{\sc Fig. 3}. A parallel-series system. }

\medskip
%\medskip

%\centerline{\includegraphics[width=10cm]{3.jpg}}

%\centerline{{\sc Figure 3}. A parallel-series system.}

%\medskip

%\begin{figure}
%\centering
%\includegraphics[width=0.7\textwidth]{3}
%\caption{A parallel-series system.}
%\end{figure}

We shall assume that $R_{i}(t)$ represents the reliability of the
$S_{i}$-th component (probability that the component $S_i$ to be functional on whole interval $[0,t]$)
in a cut set $M_{Cj}$, $j\in\{1, 2, 3, 4\}$. Therefore, there exist four possibilities
    of cut sets, and their representation shows as a \emph{parallel-series} system, as
    shown in Fig. 3, and any failure which occurs in a cut set that will cause the system failure.

    A symbolic expression for reliability of such a complex system is evaluated by applying Boolean Function (BF) Technique.
    The probability that each cut set $M_{Cj}$ fails is
$$
M_{C1}(t)=1- [(1-R_{1}(t))(1-R_{2}(t))]
$$
$$
M_{C2}(t) = 1- [(1-R_{4}(t))(1-R_{5}(t))]
$$
$$
M_{C3}(t) = 1- [(1-R_{2}(t))(1-R_{4}(t))]
$$
$$
M_{C4}(t)= 1- [(1-R_{1}(t))(1-R_{3}(t))(1-R_{5}(t))].
 $$

The cut sets are incompatible. That is why, the reliability of
the system is
$$
R_{Mc}(t)=M_{C1}(t)M_{C2}(t)M_{C3}(t)M_{C4}(t),
$$
where the computations have probabilistic-boolean sense, i.e.,
$R_i^2(t)$ is formally replaced by $R_i(t)$. We find the expression (3.1) which is the pullback
of the reliability polynomial (see also \cite{HZ,JC} and the Section 4).
\end{proof}

%============================================================================================

\section{Reliability-Geometry transfer}

The basic ingredient is a vector of probabilities $$(R_1(t),R_2(t),R_3(t),R_4(t),R_5(t))$$
and associated relation with $R_{Mc}(t)$. From the pullback we go to the polynomial equation and conversely.
Also, via a geometric interpretation, we find new properties of any reliability polynomial.

Here, geometric modeling means: (i) to change probability functions into variables
whose values are in the interval $[0,1]$, and then to variables in the interval $(-\infty,\infty)$, (ii) to analyse
and identify a body of techniques that can model certain classes
of piecewise parametric surfaces, subject to particular conditions of shape and smoothness, and (iii) to come back
into the context of probability variables, reinterpreting the geometric results.
These stochastic results are read within the (uni-, bi-,..., and six-dimensional) unit cube $[0,1], [0,1]^2$, ..., $[0,1]^6$.

\subsection{Multivariate reliability polynomial}

\begin{defn}
The multivariate polynomial
$$
R_{Mc}=R_1R_4 + R_2R_5 + R_2R_3R_4 - R_1R_2R_3R_4 \eqno(3.2)
$$
$$ - R_1R_2R_4R_5 - R_2R_3R_4R_5 + R_1R_2R_3R_4R_5
$$
is called {\it reliability polynomial}.
\end{defn}

The critical points of the polynomial ($3.2$) determine a {\it
variety} described by the system
$$\frac{\partial R}{\partial R_1}=0,\frac{\partial R}{\partial R_2}=0,
\frac{\partial R}{\partial R_3}=0,\frac{\partial R}{\partial
R_4}=0,\frac{\partial R}{\partial R_5}=0.$$
To solve this system, we can use Maple or Matlab procedures
    $$\begin{array}{lr}
        \hbox{solve}(\{
        &x_4 - x_2x_3x_4 - x_2x_4x_5 + x_2x_3x_4x_5 = 0,\mm
        &x_5 + x_3x_4 - x_1x_3x_4 - x_1x_4x_5 - x_3x_4x_5 + x_1x_3x_4x_5 = 0,\mm
        &x_2x_4 - x_1x_2x_4 - x_2x_4x_5 + x_1x_2x_4x_5 = 0,\mm
        &x_1 + x_2x_3 - x_1x_2x_3 - x_1x_2x_5 - x_2x_3x_5 + x_1x_2x_3x_5 = 0,\mm
        &x2 - x_1x_2x_4 - x_2x_3x_4 + x_1x_2x_3x_4=0\mm
        &\},\;\;[x_1, x_2, x_3, x_4, x_5])\end{array}$$
which lead to, e.g., the obvious solutions $\{(0,0,x_3,0,0)\;|\;x_3\in \mathbb{R}\}$,
which proves the nontrivial compatibility of the system.
\begin{thm} All critical points of reliability polynomial
are saddle points. \end{thm}
\begin{proof} We compute the second order differential, which turns
out to be non-definite.

It follows that the extrema points of
interest are only on the boundary of a compact set, as for example
$[0,1]^6$.
Consequently the significant optimization problems involving the
previous polynomial are of the type $\min\max$, $\max\min$ or
optimizations with constraints. In these cases we can find
solutions in the $6$-dimensional interval $[0,1]^6$.
\end{proof}
\noindent
An example for locating such solutions, using Maple, is described below:
$$\begin{array}{l}>\hbox{with(Optimization)};\mm
    >\hbox{Minimize}(x1*x4+x2*x5+x2*x3*x4-x1*x2*x3*x4\mm
    \qquad\qquad\qquad-x1*x2*x4*x5-x2*x3*x4*x5+x1*x2*x3*x4*x5,\mm
    \qquad\qquad\qquad {0. <= x1-2*x2},\hbox{assume = nonnegative)};
\end{array}$$
$$\begin{array}{l}
    \begin{array}{l}[0., [x1 = 0.684438040345821397e-1, x2 = 0., x3 = 1., x4 = 0.,\\
        \qquad\qquad x5 = .913400576368876061]];\end{array}\mm
    \begin{array}{l}[ 0.0,[x1= 0.0684438040345821397, x2= 0.0, x3= 1.0, x4= 0.0,\\
        \qquad\qquad x5= 0.913400576368876061]\end{array}\end{array}$$
$$\begin{array}{l}> \hbox{Maximize}(x1*x4+x2*x5+x2*x3*x4-x1*x2*x3*x4\\
    \qquad\qquad-x1*x2*x4*x5 -x2*x3*x4*x5+x1*x2*x3*x4*x5,\\
    \qquad\;\;\;\qquad x1 = 0 .. 1, x2 = 0 .. 1, x3 = 0 .. 1, x4 = 0 .. 1, x5 = 0 .. 1,\hbox{location})\end{array}$$
\subsection{Restrictions of reliability polynomial via \\equi-reliable components}
\begin{thm} There are $1, 15, 50, 60, 120$ diagonal
polynomials induced by 3.2, corresponding to $1, ... , 5$ variables.
\end{thm}
\begin{proof} These restrictions are counted as:

$1$ variable: $1$ polynomial.

$2$ variables: if one variable is $x$, and the other four are $y$,
we have five polynomials; if two variables are $x$ and the other
three are $y$, we have $C_5^2=10$ polynomials.

$3$ variables: if one variable is $x$, another is $y$ and the
other three are $z$, the we have $20$ polynomials; if one variable
is $x$, other two are $y$, and other two $z$, then we have $30$
polynomials.

$4$ variables: if one variable is $x$, another is $y$, another is
$z$ and the other two are $w$, then we have $60$ polynomials.

$5$ variables: the number of polynomials is permutations of $5$,
i.e. $120$.

For example, if we take $R_{1}=R_{2}= ...= R_{5}= x$, with
independent identical units, we get one diagonal (univariate) polynomial
\begin{equation}
y=2x^2+ x^3- 3x^4+x^5.
\end{equation}

Let $a$ be a parameter. For $a= \min y$, we have a double
positive root; for $a \in (\min y, 0)$, we have two positive
roots; for  $a=0$, we have three roots, zero and two positive
roots; for $a \in (0,  \max y)$, we have three positive roots; for
$a=\max y$, we have three roots, $y_{\max}$ as double and one
positive root; for $a$ greater than $\max y$, only one positive
root. \end{proof}

\begin{pro} The graph of the restriction of the polynomial
    $$y=2x^2+ x^3- 3x^4+x^5$$
to $[0,1]$ looks like "standard logistic sigmoid function graph" and particularly
like "{stress-strain curve for low-carbon steel" }.
\end{pro}

In the case of two variables, the following particular cases appear: \par\smallskip\noindent
(i) the substitutions $R_1=x, R_2=R_3=R_4=R_5=y$ produce the diagonal polynomial (there are $5$ such polynomials):
    $$P(x,y)=xy - 2xy^3 + xy^4 + y^2 + y^3 - y^4;$$
(ii) the substitutions $R_1=R_2=x, R_3=R_4=R_5=y$ produce the diagonal polynomial (there are $C_5^2=10$
    such polynomials)
$$Q(x,y)= x^2y^3 - 2x^2y^2 - xy^3 + xy^2 + 2xy.$$
\subsection{Straight lines contained in the reliability hypersurface}

The graph of the reliability polynomial is a hypersurface in $R^6$
of Cartesian explicit equation 3.2, called {\it reliability
hypersurface}. Our aim is to solve the following problem:

{\bf Problem.} How should the components $R_i$ linearly depend on time,
so that the reliability of the system be linear in time?
Geometrically, this means to find all the straight lines
which are contained in the reliability hypersurface.
\begin{thm}
The family of straight lines in the reliability
hypersurface depends on at least five and at most six parameters.
\end{thm}
\begin{proof} Let us find the number of essential parameters such
that the family of straight lines
$$R_1=a_1 t+b_1, R_2=a_2 t+b_2, R_3=a_3 t+b_3,$$
$$ R_4=a_4 t+b_4, R_5=a_5 t+b_5, R(S)= a_6t+b_6,$$
with $a_1^2+...+a_6^2>0$, are included in reliability hypersurface. Here,
$t$ is a parameter on the line. All reasonings remain similar if we
replace $a_it+b_i$ by $a_ie^{-b_i t},\, a_i, \,b_i >0$.

First, we compute the following products:
\begin{multline*}\begin{split}
    R_1R_4=&a_1a_4 t^2 + (a_1b_4+b_1a_4)t + b_1b_4,\\
    R_2R_5=&a_2a_5 t^2 + (a_2b_5+b_2a_5)t + b_2b_5,\\
    R_2R_3R_4=&(a_2a_3a_4)t^3 + (a_2a_3b_4 + a_2a_4b_3 + a_3a_4b_2)t^2\\&+ (a_2b_3b_4 + a_3b_2b_4
         + a_4b_2b_3)t + b_2b_3b_4,\\
    R_1R_2R_3R_4 =
        & (a_1a_2a_3a_4)t^4 + (a_1a_2a_3b_4 + a_1a_2a_4b_3 +a_1a_3a_4b_2\\& + a_2a_3a_4b_1)t^3
         + (a_1a_2b_3b_4 + a_1a_3b_2b_4+ a_1a_4b_2b_3\\
        &+ a_2a_3b_1b_4 + a_2a_4b_1b_3 + a_3a_4b_1b_2)t^2\\
        &+ (a_1b_2b_3b_4 + a_2b_1b_3b_4 + a_3b_1b_2b_4 + a_4b_1b_2b_3)t\\&+b_1b_2b_3b_4,
         \end{split}\end{multline*}
  \begin{multline*}\begin{split}
        R_1R_2R_4R_5=
        &(a_1a_2a_4a_5)t^4 + (a_1a_2a_4b_5 + a_1a_2a_5b_4 +a_1a_4a_5b_2\\& + a_2a_4a_5b_1)t^3
         + (a_1a_2b_4b_5 + a_1a_4b_2b_5+ a_1a_5b_2b_4\\&+ a_2a_4b_1b_5 + a_2a_5b_1b_4+ a_4a_5b_1b_2)t^2+ (a_1b_2b_4b_5 \\&
        + a_2b_1b_4b_5 + a_4b_1b_2b_5 + a_5b_1b_2b_4)t+ b_1b_2b_4b_5,
     \end{split}\end{multline*}

\begin{multline*}\begin{split}
    R_2R_3R_4R_5 =& (a_2a_3a_4a_5)t^4 + (a_2a_3a_4b_5 + a_2a_3a_5b_4 \\
    &+a_2a_4a_5b_3 + a_3a_4a_5b_2)t^3\\
        & + (a_2a_3b_4b_5 + a_2a_4b_3b_5+ a_2a_5b_3b_4 \\
        &+ a_3a_4b_2b_5 + a_3a_5b_2b_4+ a_4a_5b_2b_3)t^2\\
        & + (a_2b_3b_4b_5 + a_3b_2b_4b_5 + a_4b_2b_3b_5 \\&+ a_5b_2b_3b_4)t+ b_2b_3b_4b_5,\\
    R_1R_2R_3R_4R_5=&(a_1a_2a_3a_4a_5)t^5 + (a_1a_2a_3a_4b_5 +a_1a_2a_3a_5b_4 + a_1a_2a_4a_5b_3\\
        & + a_1a_3a_4a_5b_2 +a_2a_3a_4a_5b_1)t^4 + (a_1a_2a_3b_4b_5 + a_1a_2a_4b_3b_5\\
        & +a_1a_2a_5b_3b_4 + a_1a_3a_4b_2b_5 + a_1a_3a_5b_2b_4 +a_1a_4a_5b_2b_3\\
        & + a_2a_3a_4b_1b_5 + a_2a_3a_5b_1b_4 +a_2a_4a_5b_1b_3 + a_3a_4a_5b_1b_2)t^3\\
        & + (a_1a_2b_3b_4b_5 +a_1a_3b_2b_4b_5 + a_1a_4b_2b_3b_5 + a_1a_5b_2b_3b_4\\
        & +a_2a_3b_1b_4b_5 + a_2a_4b_1b_3b_5 + a_2a_5b_1b_3b_4 +a_3a_4b_1b_2b_5\\
        & + a_3a_5b_1b_2b_4 + a_4a_5b_1b_2b_3)t^2 +(a_1b_2b_3b_4b_5 + a_2b_1b_3b_4b_5\\
        & + a_3b_1b_2b_4b_5 +a_4b_1b_2b_3b_5 + a_5b_1b_2b_3b_4)t + b_1b_2b_3b_4b_5.
\end{split}\end{multline*}
%===
    By replacement, ordering by powers of $t$ and identifying,
    we obtain a system whose solutions describe the number of straight lines situated on the reliability hypersurface.
    We write the system ordering by the coefficients of the powers of degree from zero to five, relative to $t$:

\begin{multline*}\begin{split}
    b_6=&b_1b_4 + b_2b_5 + b_2b_3b_4 - b_1b_2b_3b_4 - b_1b_2b_4b_5 -b_2b_3b_4b_5 + b_1b_2b_3b_4b_5, \\
    a_6=&a_1b_4 + a_4b_1 + a_2b_5 +a_5b_2 + a_2b_3b_4 + a_3b_2b_4 + a_4b_2b_3 - a_1b_2b_3b_4\\
        & -a_2b_1b_3b_4 - a_3b_1b_2b_4 - a_4b_1b_2b_3 - a_1b_2b_4b_5 -a_2b_1b_4b_5 - a_4b_1b_2b_5\\
        & - a_5b_1b_2b_4 - a_2b_3b_4b_5 -a_3b_2b_4b_5 - a_4b_2b_3b_5 - a_5b_2b_3b_4 + a_1b_2b_3b_4b_5\\
        & +a_2b_1b_3b_4b_5 + a_3b_1b_2b_4b_5 + a_4b_1b_2b_3b_5 +a_5b_1b_2b_3b_4;\\
    0=&a_1a_4 + a_2a_5 + a_2a_3b_4 + a_2a_4b_3 + a_3a_4b_2 -a_1a_2b_3b_4 - a_1a_3b_2b_4\\
             &- a_1a_4b_2b_3 - a_2a_3b_1b_4 -a_2a_4b_1b_3 - a_3a_4b_1b_2 \\
             &- a_1a_2b_4b_5 - a_1a_4b_2b_5 -a_1a_5b_2b_4\\
        & - a_2a_4b_1b_5 - a_2a_5b_1b_4 - a_4a_5b_1b_2\\
        & -a_2a_3b_4b_5 - a_2a_4b_3b_5 - a_2a_5b_3b_4
\end{split}\end{multline*}
\begin{multline*}\begin{split}
        & - a_3a_4b_2b_5 -a_3a_5b_2b_4 - a_4a_5b_2b_3 \\
        &+ a_1a_2b_3b_4b_5 + a_1a_3b_2b_4b_5 +a_1a_4b_2b_3b_5\\
        & + a_1a_5b_2b_3b_4 + a_2a_3b_1b_4b_5 +a_2a_4b_1b_3b_5 \\
        &+ a_2a_5b_1b_3b_4 + a_3a_4b_1b_2b_5\\
        & +a_3a_5b_1b_2b_4 + a_4a_5b_1b_2b_3;\\
    0=&a_2a_3a_4 - a_1a_2a_3b_4 - a_1a_2a_4b_3 - a_1a_3a_4b_2 \\
    &-a_2a_3a_4b_1 - a_1a_2a_4b_5 - a_1a_2a_5b_4\\
        & - a_1a_4a_5b_2 -a_2a_4a_5b_1 - a_2a_3a_4b_5 - a_2a_3a_5b_4\\
        & - a_2a_4a_5b_3 -a_3a_4a_5b_2\\
        & + a_1a_2a_3b_4b_5 + a_1a_2a_4b_3b_5 +a_1a_2a_5b_3b_4 \\
        &+ a_1a_3a_4b_2b_5 + a_1a_3a_5b_2b_4\\
        & +a_1a_4a_5b_2b_3 + a_2a_3a_4b_1b_5 + a_2a_3a_5b_1b_4 \\
        &+a_2a_4a_5b_1b_3 + a_3a_4a_5b_1b_2;\\
    0=&a_1a_2a_3a_4b_5 - a_1a_2a_4a_5 - a_2a_3a_4a_5 - a_1a_2a_3a_4\\
    & + a_1a_2a_3a_5b_4 + a_1a_2a_4a_5b_3\\
        & + a_1a_3a_4a_5b_2 + a_2a_3a_4a_5b_1,\\
         0=&a_{1}a_{2}a_{3}a_{4}a_{5}.\end{split}\end{multline*}
Starting from the last equation, at least one of the numbers $a_i,
i=1,...,5$ must be zero (number of cases:
$C_5^1+C_5^2+C_5^3+C_5^4=30$). So the straight-lines are parallel
to some hyperplane of coordinates. The first equation shows that
at $t=0$, the point $(b_1,...,b_6)$ is on the reliability
hypersurface. This remark requires the following procedure: we
choose arbitrarily $b_1,...,b_5$, and compute $b_6$. We replace the
values $b_1,...,b_5$ in the remaining equations. If the new system,
in unknown $(a_1,...,a_6)$, has a solution with
at least non-zero component, then there exists one straight line
passing through the point $(b_1,...,b_6)$ and lying on the
reliability hypersurface.
Explicitly, after solving the algebraic system, we have the following cases:\par\smallskip\noindent
%===
{\small
\textbf{Case 1} ($a_1=0$):
\begin{multline*}\begin{split}
    b_6 =&b_1b_4 + b_2b_5 + b_2b_3b_4 - b_1b_2b_3b_4 - b_1b_2b_4b_5\\
    & -b_2b_3b_4b_5 + b_1b_2b_3b_4b_5,\\
    a_6=& a_4b_1 + a_2b_5 + a_5b_2 + a_2b_3b_4 + a_3b_2b_4 + a_4b_2b_3\\
            &- a_2b_1b_3b_4 - a_3b_1b_2b_4  - a_4b_1b_2b_3- a_2b_1b_4b_5 \\
            &-a_4b_1b_2b_5 - a_5b_1b_2b_4 - a_2b_3b_4b_5 - a_3b_2b_4b_5\\
        & -a_4b_2b_3b_5 - a_5b_2b_3b_4 + a_2b_1b_3b_4b_5\\
        & + a_3b_1b_2b_4b_5 +a_4b_1b_2b_3b_5+ a_5b_1b_2b_3b_4,
     \end{split}\end{multline*}
    \begin{multline*}\begin{split}
    0 =& a_2a_5 + a_2a_3b_4 + a_2a_4b_3 + a_3a_4b_2\\
    & - a_2a_3b_1b_4 -a_2a_4b_1b_3 - a_3a_4b_1b_2\\
        & - a_2a_4b_1b_5 - a_2a_5b_1b_4 -a_4a_5b_1b_2\\
        & - a_2a_3b_4b_5 - a_2a_4b_3b_5 - a_2a_5b_3b_4\\
        & -a_3a_4b_2b_5 - a_3a_5b_2b_4 - a_4a_5b_2b_3 \\
        &+ a_2a_3b_1b_4b_5 +a_2a_4b_1b_3b_5 + a_2a_5b_1b_3b_4\\
        & + a_3a_4b_1b_2b_5 +a_3a_5b_1b_2b_4 + a_4a_5b_1b_2b_3,\\
    0 =&-a_2a_3a_4 - a_2a_3a_4b_1 - a_2a_4a_5b_1 \\
    &- a_2a_3a_4b_5 -a_2a_3a_5b_4 - a_2a_4a_5b_3 - a_3a_4a_5b_2\\
        &  + a_2a_3a_4b_1b_5 +a_2a_3a_5b_1b_4 + a_2a_4a_5b_1b_3 + a_3a_4a_5b_1b_2,\\
            0=&a_2a_3a_4a_5(b_1- 1).
            \end{split}\end{multline*}
i) ($a_1=0$ and $b_1=1$):
\begin{multline*}\begin{split}
    b_6 =&b_4 + b_2b_5 - b_2b_4b_5,\\
    a_6=& a_4 + a_2b_5 + a_5b_2 - a_2b_4b_5 - a_4b_2b_5 - a_5b_2b_4,\\
    0 =& a_2a_5 - a_2a_4b_5 - a_2a_5b_4 - a_4a_5b_2,\\
    0 =&a_2a_4a_5.\end{split}\end{multline*}
    In this case, for an arbitrary point $(b_1=1,b_2,b_3,b_4,b_5,b_6)$,
    the solution $(a_1=0,a_2,a_3$, $a_4,a_5,a_6)$ depends on six parameters (a family of straight lines).
    All the foregoing straight lines are in the plane $R_1=1$.
    In this case the reliability hypersurface is a fiber bundle (ruled hypersurface).
ii) ($a_1=a_2=0$):
\begin{multline*}\begin{split}
    b_6 =&b_1b_4 + b_2b_5 + b_2b_3b_4 - b_1b_2b_3b_4 - b_1b_2b_4b_5\\
    & -b_2b_3b_4b_5 + b_1b_2b_3b_4b_5,\\
    a_6=& a_4b_1 + a_5b_2 + a_3b_2b_4 + a_4b_2b_3 - a_3b_1b_2b_4\\
    & -a_4b_1b_2b_3 - a_4b_1b_2b_5- a_5b_1b_2b_4 \\
        & - a_3b_2b_4b_5 -a_4b_2b_3b_5 - a_5b_2b_3b_4 \\
        &+ a_3b_1b_2b_4b_5 + a_4b_1b_2b_3b_5+ a_5b_1b_2b_3b_4,\\
    0 =& a_3a_4b_2 - a_3a_4b_1b_2 - a_4a_5b_1b_2 \\
    &- a_3a_4b_2b_5 -a_3a_5b_2b_4 - a_4a_5b_2b_3 \\
        &+ a_3a_4b_1b_2b_5 + a_3a_5b_1b_2b_4+ a_4a_5b_1b_2b_3,\\
        0 =&a_3a_4a_5b_2(b_1- 1).
        \end{split}\end{multline*}
iii) ($a_1=a_3=0$):
\begin{multline*}\begin{split}
    b_6 =&b_1b_4 + b_2b_5 + b_2b_3b_4 - b_1b_2b_3b_4 \\
    &- b_1b_2b_4b_5 -b_2b_3b_4b_5 + b_1b_2b_3b_4b_5,\\
    a_6=& a_4b_1 + a_2b_5 + a_5b_2 + a_2b_3b_4 + a_4b_2b_3\\
    & -a_2b_1b_3b_4 - a_4b_1b_2b_3\\
\end{split}\end{multline*}
\begin{multline*}\begin{split}
        & - a_2b_1b_4b_5 - a_4b_1b_2b_5 -a_5b_1b_2b_4 \\
        &- a_2b_3b_4b_5 - a_4b_2b_3b_5 - a_5b_2b_3b_4\\
        & +a_2b_1b_3b_4b_5 + a_4b_1b_2b_3b_5 + a_5b_1b_2b_3b_4,\\
    0 =& a_2a_5 + a_2a_4b_3 - a_2a_4b_1b_3 - a_2a_4b_1b_5\\
    & -a_2a_5b_1b_4 - a_4a_5b_1b_2 - a_2a_4b_3b_5\\
        & - a_2a_5b_3b_4 -a_4a_5b_2b_3 + a_2a_4b_1b_3b_5\\
        & + a_2a_5b_1b_3b_4 + a_4a_5b_1b_2b_3,\\
    0 =&a_2a_4a_5(b_1b_3 - b_3 - b_1).
    \end{split}\end{multline*}

iv) ($a_1a_4=0$):
\begin{multline*}\begin{split}
    b_6 =&b_1b_4 + b_2b_5 + b_2b_3b_4 - b_1b_2b_3b_4\\
    & - b_1b_2b_4b_5 -b_2b_3b_4b_5 + b_1b_2b_3b_4b_5,\\
    a_6=& a_2b_5 + a_5b_2 + a_2b_3b_4 + a_3b_2b_4 - a_2b_1b_3b_4\\
    & -a_3b_1b_2b_4 - a_2b_1b_4b_5- a_5b_1b_2b_4\\
        &  - a_2b_3b_4b_5 -a_3b_2b_4b_5 - a_5b_2b_3b_4 + a_2b_1b_3b_4b_5\\
        & + a_3b_1b_2b_4b_5+ a_5b_1b_2b_3b_4,\\
    0 =& a_2a_5 + a_2a_3b_4 - a_2a_3b_1b_4 - a_2a_5b_1b_4\\
    & -a_2a_3b_4b_5 - a_2a_5b_3b_4- a_3a_5b_2b_4\\
        & + a_2a_3b_1b_4b_5 +a_2a_5b_1b_3b_4 + a_3a_5b_1b_2b_4,\\
            0=&a_2a_3a_5b_4(b_1-1).
            \end{split}\end{multline*}
v) ($a_1a_5=0$):
\begin{multline*}\begin{split}
    b_6 =&b_1b_4 + b_2b_5 + b_2b_3b_4 - b_1b_2b_3b_4\\
    & - b_1b_2b_4b_5 -b_2b_3b_4b_5 + b_1b_2b_3b_4b_5,\\
    a_6=& a_4b_1 + a_2b_5 + a_2b_3b_4 + a_3b_2b_4 \\
    &+ a_4b_2b_3 -a_2b_1b_3b_4 - a_3b_1b_2b_4\\
        & - a_4b_1b_2b_3 - a_2b_1b_4b_5 -a_4b_1b_2b_5 \\
        &- a_2b_3b_4b_5 - a_3b_2b_4b_5 - a_4b_2b_3b_5\\
        & +a_2b_1b_3b_4b_5 + a_3b_1b_2b_4b_5 + a_4b_1b_2b_3b_5,\\
    0 =& a_2a_3b_4 + a_2a_4b_3 + a_3a_4b_2 - a_2a_3b_1b_4 \\
    &-a_2a_4b_1b_3 - a_3a_4b_1b_2 - a_2a_4b_1b_5\\
        & - a_2a_3b_4b_5 -a_2a_4b_3b_5 - a_3a_4b_2b_5 \\
        &+ a_2a_3b_1b_4b_5 + a_2a_4b_1b_3b_5 + a_3a_4b_1b_2b_5,\\
    0 =&a_2a_3a_4(1-b_1-b_5+b_1b_5).
    \end{split}\end{multline*}
    In case v), for an arbitrary point $(b_1,b_2,b_3,b_4,b_5,b_6)$,
    the solution $(a_1=0,a_2,a_3$, $a_4,a_5=0,a_6)$ depends on six parameters (a family of straight lines).
    In this case the reliability hypersurface is a fiber bundle (ruled hypersurface).
    The situations ii)-iv) are similar.
\textbf{Case 2.} ($a_1=0, a_2=0$):
\begin{multline*}\begin{split}
    b_6 =&b_1b_4 + b_2b_5 + b_2b_3b_4 - b_1b_2b_3b_4 \\
    &- b_1b_2b_4b_5 -b_2b_3b_4b_5 + b_1b_2b_3b_4b_5,\\
    a_6=& a_4b_1 + a_5b_2 + a_3b_2b_4 + a_4b_2b_3 \\
    &- a_3b_1b_2b_4 -a_4b_1b_2b_3 - a_4b_1b_2b_5- a_5b_1b_2b_4\\
        &  - a_3b_2b_4b_5 -a_4b_2b_3b_5 - a_5b_2b_3b_4 \\
        &+ a_3b_1b_2b_4b_5 + a_4b_1b_2b_3b_5+ a_5b_1b_2b_3b_4,\\
    0 =& a_3a_4b_2 - a_3a_4b_1b_2 - a_4a_5b_1b_2 \\
    &- a_3a_4b_2b_5 -a_3a_5b_2b_4 - a_4a_5b_2b_3\\
        & + a_3a_4b_1b_2b_5 + a_3a_5b_1b_2b_4+ a_4a_5b_1b_2b_3,\\
             0 =&a_3a_4a_5b_2(b_1- 1).
\end{split}\end{multline*}
i) ($a_1=0,a_2=0$ and $b_1=1$):
\begin{multline*}\begin{split}
    b_6 =&b_4 + b_2b_5 - b_2b_4b_5,\\
    a_6=& a_4 + a_5b_2 - a_4b_2b_5 - a_5b_2b_4,\\
    0 =& a_2a_5 - a_2a_4b_5 - a_2a_5b_4 - a_4a_5b_2,\\
    0 =&a_4a_5b_2 .
    \end{split}\end{multline*}
ii) ($a_1=0,a_2=0, a_3=0$):
\begin{multline*}\begin{split}
    b_6 =&b_1b_4 + b_2b_5 + b_2b_3b_4 - b_1b_2b_3b_4 \\
    &- b_1b_2b_4b_5 -b_2b_3b_4b_5 + b_1b_2b_3b_4b_5,\\
    a_6=& a_4 + a_5b_2 - a_4b_2b_5 - a_5b_2b_4,\\
    0 =& a_4b_1 + a_5b_2 + a_4b_2b_3 - a_4b_1b_2b_3 \\
    &- a_4b_1b_2b_5 -a_5b_1b_2b_4 - a_4b_2b_3b_5- a_5b_2b_3b_4\\
        &  + a_4b_1b_2b_3b_5 +a_5b_1b_2b_3b_4\\
            0 =&a_4a_5(b_1b_2b_3 - b_2b_3 - b_1b_2).
            \end{split}\end{multline*}
iii) ($a_1=0,a_2=0, a_4=0$):
\begin{multline*}\begin{split}
    b_6 =&b_1b_4 + b_2b_5 + b_2b_3b_4 - b_1b_2b_3b_4\\
    & - b_1b_2b_4b_5 -b_2b_3b_4b_5 + b_1b_2b_3b_4b_5,\\
    a_6=& a_5b_2 + a_3b_2b_4 - a_3b_1b_2b_4 \\
    &- a_5b_1b_2b_4 -a_3b_2b_4b_5 - a_5b_2b_3b_4\\
        & + a_3b_1b_2b_4b_5 + a_5b_1b_2b_3b_4,\\
           0 =&a_3a_5b_2b_4(b_1-1).
           \end{split}\end{multline*}
\textbf{Case 3.} ($a_1=0, a_2=0, a_3=0$):
\begin{multline*}\begin{split}
    b_6 =&b_1b_4 + b_2b_5 + b_2b_3b_4 \\
    &- b_1b_2b_3b_4 - b_1b_2b_4b_5 - b_2b_3b_4b_5 + b_1b_2b_3b_4b_5,\\
    a_6=& a_4b_1 + a_5b_2 + a_4b_2b_3 - a_4b_1b_2b_3 - a_4b_1b_2b_5 -
    a_5b_1b_2b_4 - a_4b_2b_3b_5\\& - a_5b_2b_3b_4 + a_4b_1b_2b_3b_5 +
    a_5b_1b_2b_3b_4,\\
    0 =&a_4a_5(b_1b_2b_3-b_2b_3-b_1b_2).
    \end{split}\end{multline*}
i) ($a_1=0,a_2=0,a_3=0, a_4=0$):
\begin{multline*}\begin{split}
    b_6 =&b_1b_4 + b_2b_5 + b_2b_3b_4 \\
    &- b_1b_2b_3b_4 - b_1b_2b_4b_5 - b_2b_3b_4b_5 + b_1b_2b_3b_4b_5,\\
    a_6=&a_5b_2(1-b_1b_4-b_3b_4+b_1b_3b_4).
    \end{split}\end{multline*}
ii) ($a_1=0,a_2=0,a_3=0, a_5=0$):
\begin{multline*}\begin{split}
    b_6 =&b_1b_4 + b_2b_5 + b_2b_3b_4 - b_1b_2b_3b_4 - b_1b_2b_4b_5 -
    b_2b_3b_4b_5 + b_1b_2b_3b_4b_5,\\
    a_6=& a_4b_1 + a_4b_2b_3 - a_4b_1b_2b_3 - a_4b_1b_2b_5 -
    a_4b_2b_3b_5 + a_4b_1b_2b_3b_5.
    \end{split}\end{multline*}
\textbf{Case 4.} ($a_1=0, a_2=0, a_3=0, a_4=0$):
\begin{multline*}\begin{split}
b_6 =&b_1b_4 + b_2b_5 + b_2b_3b_4 - b_1b_2b_3b_4 - b_1b_2b_4b_5 -
b_2b_3b_4b_5 + b_1b_2b_3b_4b_5,\\
a_6=& a_5b_2(1 - b_1b_4 - b_3b_4 + b_1b_3b_4).
\end{split}\end{multline*}
The rest of cases are similar.
}
%===
For each case, using the Jacobian matrix and its rank, we count the number of essential parameters.
\end{proof}
\subsection{Returning to the probability framework}
In order to return to the probability ansatz, we must assume that the
coefficients $a_i, b_i,i=1,...,6$ satisfy the conditions imposed by the
assumption that each function $a_i t + b_i$, $i=1,...,6,$ is a probability,
i.e., $0\leq a_i t + b_i \leq 1, i=1,...,6$. If $a_k=0$, then $0\leq b_k\leq 1$.
We further assume that all $a_i$ are different from zero.
(i) If $a_i>0$, then we find the intervals\ $I_i:\,\,-\frac{b_i}{a_i}\leq t \leq \frac{1-b_i}{a_i}, i=1,...,6$.
(ii) If there exists $a_k<0$, then a non-void interval is $I_k:\,\, \frac{1-b_k}{a_k}\leq t\leq -\frac{b_k}{a_k}$.

Suppose we have a non-void intersection $I=\cap I_i$.
Consequently, the significant parts from probabilistic point of view are
segments of straight lines included in the interval $[0,1]^6$.
\begin{thm}
Let us consider the vector of probabilities
$$(R_1(t),R_2(t),R_3(t),R_4(t),R_5(t)).$$
The most plausible situation is that
which imposes a maximum number of parameters in the family of
straight lines on the reliability hypersurface. \end{thm}
\begin{proof} In this case we have maximum degrees of freedom
(number of parameters). \end{proof}
\begin{rem} We can reiterate the process, by replacing this time the affine
framework with an exponential or a logarithmic one.
\end{rem}
\subsection{Equi-reliable hypersurfaces}
We further consider in $\mathbb{R}^5$ the constant level algebraic hypersurfaces
of the reliability polynomial (the {\it equi-reliable hypersurfaces}):
\begin{equation}\label{1c}\begin{split}c=&R_1R_4 + R_2R_5 + R_2R_3R_4- R_1R_2R_3R_4\\
    & - R_1R_2R_4R_5 -R_2R_3R_4R_5 + R_1R_2R_3R_4R_5.\end{split}\end{equation}
{\bf Open problem.} How many straight lines are included in each equi-reliable hypersurface?
As an example, the constant level zero hypersurface contains the linear varieties
    $OR_3R_4R_5: R_1=0,R_2=0$; $OR_1R_3R_5: R_2=0,R_4=0$; $OR_2R_3: R_1=0, R_4=0,R_5=0$.
Indeed, we have
    $$R_{Mc}=R_1(R_4 - R_2R_3R_4 - R_2R_4R_5+R_2R_3R_4R_5)+ R_2(R_5 + R_3R_4- R_3R_4R_5).$$
{\bf Acknowledgements.} Partially supported by University Politehnica of Bucharest and by Academy of Romanian Scientists.

The first author would like to acknowledge the financial support of the Ph.D. studies
by the Iraqi Ministry of Higher Education and Scientific Research, and to thank to
Professor Fouad A. Majeed from University of Babylon for the helpful discussions.

PhD student Zahir Abdul Haddi Hassan\\
Department of Mathematics and Informatics,\\
Faculty of Applied Sciences,\\
University {\sc Politehnica} of Bucharest, \\Splaiul Independentei
313, RO-060042, Bucharest, Romania; \\Department of Mathematics, College for Pure Sciences,\\
University of {\sc Babylon}, Babylon, Iraq. \\E-mail: {\tt
zaher\_haddi@yahoo.com}, {\tt mathzahir@gmail.com}\\

\vspace{0.5cm}
Prof. Emeritus Dr. Constantin {\sc Udri\c ste}\\
Department of Mathematics and Informatics, \\Faculty of Applied
Sciences, \\University {\sc Politehnica} of Bucharest, \\Splaiul
Independentei 313, RO-060042, Bucharest, Romania.\\
   E-mail: {\tt udriste@mathem.pub.ro}\\

\vspace{0.5cm}
Prof. Dr. Vladimir {\sc Balan}\\
Department of Mathematics and Informatics, \\Faculty of Applied
Sciences, \\University {\sc Politehnica} of Bucharest, \\Splaiul
Independentei 313, RO-060042, Bucharest, Romania.\\
E-mail: {\tt vbalan@mathem.pub.ro}

\end{document}